\documentclass[12pt]{amsart}
\usepackage{latexsym}
\usepackage{amsthm}
\usepackage{amsmath}
\usepackage{amsfonts}
\usepackage{amssymb}
\usepackage[dvips]{graphicx}
\usepackage{xypic}
\input xy
\xyoption{all}

\newtheorem{theo}{Theorem}[section]
\newtheorem{lemma}[theo]{Lemma}

\newtheorem{defi}[theo]{Definition}
\newtheorem{coro}[theo]{Corollary}
\newtheorem{rem}[theo]{Remark}
\newtheorem{exam}[theo]{Example}

\newcommand\Mod{\operatorname{Mod}}
\newcommand\HMod{\operatorname{HMod}}

\newcommand\op{\operatorname{op}}
\newcommand\fp{\operatorname{fp}}
\newcommand\cfp{\operatorname{cfp}}

\newcommand\Id{\operatorname{Id}}
\newcommand\Ex{\operatorname{Ex}}

\newcommand\SSet{\operatorname{\bf SSet}}

\newcommand\Cat{\operatorname{\bf Cat}}
\newcommand\Ho{\operatorname{Ho}}

\newcommand\dom{\operatorname{dom}}

\newcommand\cod{\operatorname{cod}}

\newcommand\cc{\mathcal {C}}
\newcommand\cd{\mathcal {D}}

\newcommand\ch{\mathcal {H}}

\newcommand\ck{\mathcal {K}}
\newcommand\cl{\mathcal {L}}
\newcommand\cm{\mathcal {M}}

\newcommand\ct{\mathcal {T}}

\newcommand\cv{\mathcal {V}}

\date{May 9, 2013}
 
\begin{document}
\title[Rigidification of algebras]
{Rigidification of algebras over essentially algebraic theories}
\author[J. Rosick\'{y}]
{J. Rosick\'{y}$^*$}
\thanks{ $^*$ Supported by the Grant Agency of the Czech republic under grant 201/11/0528.} 
\address{
\newline 
Department of Mathematics and Statistics\newline
Masaryk University, Faculty of Sciences\newline
Kotl\'{a}\v{r}sk\'{a} 2, 611 37 Brno, Czech Republic\newline
rosicky@math.muni.cz
}
 
\begin{abstract}
Badzioch and Bergner proved a rigidification theorem saying that each homotopy simplicial algebra is weakly equivalent to a simplicial algebra. 
The question is whether this result can be extended from algebraic theories to finite limit theories and from simplicial sets to more general monoidal 
model categories. We will present some answers to this question.
\end{abstract}

\maketitle
 
\section{Introduction}
Badzioch \cite{Ba} proved a rigidification theorem for simplicial algebras of one-sorted algebraic theories $\ct$ saying that any homotopy
$\ct$-algebra is weakly equivalent to a (strict) $\ct$-algebra. Bergner \cite{Be} extended this rigidification theorem to (many-sorted) algebraic
theories. Our aim is to find whether their rigidification theorems can be generalized to an arbitrary finitely combinatorial monoidal model category 
$\cv$ in place of simplicial sets and to a finite weighted limit theory $\ct$ in place of an algebraic theory. These theories are usually called 
essentially algebraic (see \cite{AR}). In the homotopy context, we have to work with weighted limits whose weight is cofibrant (see \cite{LR}, 
or \cite{V}). Since, in contrast to finite products, finite weights are rarely cofibrant, we have to replace finite weights by their saturation 
consisting of finitely presentable weights. Then we can use finitely presentable cofibrant weights to define finite weighted homotopy limit theories.  

The rigidification theorem of \cite{Ba} and \cite{Be} has a strong form saying that the model categories of strict algebras and of homotopy algebras 
are Quillen equivalent. We will show that this strong form always fo\-llows from a weak one and is valid for $\ct$ having all limits weighted 
by a suitable class $\Phi$ of finitely presentable cofibrant weights. The condition is that any cofibrant weight can be obtained from $\Phi$-weights 
by means of homotopy invariant $\Phi$-flat colimits. In particular, we can take all finitely presentable cofibrant weights because, in a finitely 
combinatorial model category, any cofibrant object is a filtered colimit of finitely presentable cofibrant objects. Or, we can take all finite products, 
i.e., an algebraic theory, provided that any cofibrant weight is a homotopy sifted colimit of finite coproducts of representables. On the other hand, 
we will show that the rigidification theorem is not always true and that it means a kind of coherence statement.  
 
We will need some assumptions about $\cv$, above all it should be a monoidal model category in the sense of \cite{L}, i.e., with the cofibrant
unit $I$ (\cite{Ho} has this axiom in  a weaker form). This makes possible to define model $\cv$-categories (see \cite{L}). Also, $\cv$ should be 
locally finitely presentable as a closed category (see \cite{K}) and finitely combinatorial. The latter adds that both cofibrations and trivial 
cofibrations are cofibrantly generated by morphisms between finitely presentable objects. Since we need the projective $\cv$-model structure 
on $[\ct,\cv]$, the $\cv$-category $\ct$ should be locally cofibrant (i.e., it should have all hom-objects cofibrant), or $\cv$ should satisfy 
the monoid axiom. Since this projective model structure should be left proper as well, we will prefer the first assumption (see \cite{DRO}). In order 
to make the machinery of enriched left Bousfield localizations possible, we have to assume that $\cv$ is not only finitely combinatorial but finitely 
tractable, which means that the generating cofibrations and trivial cofibrations are between cofibrant objects (see \cite{B}). Even more restrictively, 
in order to prove our main results we have to assume that all objects of $\cv$ are cofibrant. Finally, will need a fibrant approximation $\cv$-functor 
on $\cv$ preserving limits weighted by finite weights. 

Concerning enriched category theory we refer to \cite{K1}. In particular, given a $\cv$-category $\ck$, a diagram $D:\cd\to\ck$ and a weight 
$G:\cd\to\cv$ then a limit $\{G,D\}$ of $D$ weighted by $G$ is defined by being equipped with a natural isomorphism
$$
\beta:\ck(-,\{G,D\})\to[\cd,\cv](G,\ck(-,D)).
$$
This natural transformation corresponds to a weighted limit cone
$$
\delta:G\to\ck(\{G,D\},D).
$$

The author is indebted to John Bourke, Richard Garner, A. E. Stanculescu and Luk\' a\v s Vok\v r\'{i}nek for stimulating discussions about the subject 
of this paper. But, in particular, the author is grateful to the unknown referee for finding a gap in the proof of \ref{th3.3} and for pointing up
the need of taking the saturation in \ref{re3.5}.

\section{Homotopy limit sketches}
We recall the concept of a weighted limit sketch (see \cite{K1}).

\begin{defi}\label{def2.1}
{\em
A \textit{weighted limit sketch} is a pair $\ch=(\ct,L)$ consisting of a small $\cv$-category $\ct$ and a set $L$ 
of weights $G_l:\cd_l\to\cv$, diagrams $D_l:\cd_l\to\ct$, objects $X_l$ and morphisms
$$
\delta_l:G_l\to\ct(X_l,D_l)
$$
in $[\ct,\cv]$ for each $l\in L$.

A \textit{model} of $\ch$ is a $\cv$-functor $A:\ct\to\cv$ such that  
$$
G_l \xrightarrow{\quad  \delta_l\quad} \ct(X_l,D_l) \xrightarrow{\quad \quad} \cv(AX_l,AD_l)
$$
is a weighted limit cone for each $l\in L$. 
}
\end{defi}

The last statement means that the induced morphisms
$$
t_l^A:AX_l\to\{G_l,AD_l\}
$$
are isomorphisms. We will denote by $\Mod(\ch)$ the full subcategory of $[\ct,\cv]$ consisting of all models of $\ch$.

A weight $G:\cd\to\cv$ is called \textit{finite} if 
\begin{enumerate}
\item[(i)]  $\cd$ has finitely many objects,
\item[(ii)]  all objects $\cd(d,e)$ are finitely presentable, and
\item[(iii)]  all objects $Gd$ are finitely presentable.
\end{enumerate}
This concept was introduced in \cite{K}. Since any finitely presentable weight belongs to the closure of representable functors under
colimits weighted by finite weights (see \cite{K}, 7.2), finitely presentable weights form the saturation of finite weights (see \cite{KS},
3.8 and 3.13). 

\begin{defi}\label{def2.2}
{\em
A weighted limit sketch is called \textit{finite} if all weights $G_l$, $l\in L$ are finitely presentable.
}
\end{defi}

\cite{K} calls a weighted limit sketch finite if all weights are finite. Our definition is more general and we will need it later. But
its strength is the same as that of \cite{K}.  

A \textit{fibrant approximation functor} $R:\cv\to\cv$ is a functor $R$ together with a natural transformation $\rho:\Id\to R$ such that
$\rho_V$ is a weak equivalence and $RV$ is fibrant for each $V\in\cv$ (cf. \cite{H}). If all $\rho_V$ are trivial cofibrations we will call
$R$ a \textit{fibrant replacement functor} (cf. \cite{Ho}).

\begin{theo}\label{th2.3}
Let $\cv$ be a combinatorial monoidal model category equipped with a fibrant approximation $\cv$-functor $R:\cv\to\cv$ preserving finite weighted 
limits. Let $\cl=(\ct,L)$ be a finite weighted limit sketch with $\ct$ locally cofibrant. Then $\Mod(\ch)$ is a combinatorial model $\cv$-category 
with respect to the projective model structure. 
\end{theo}
\begin{proof}
Following \cite{K}, $\Mod(\cl)$ is a reflective subcategory of $[\ct,\cv]$. We will denote the inclusion $\cv$-functor by
$U:\Mod(\ch)\to\ [\ct,\cv]$ and its left $\cv$-adjoint by $F$. Since any finitely presentable weight belongs to the closure of representable 
functors under colimits weighted by finite weights, the fibrant approximation functor $R$ preserves limits weighted by finitely presentable
weights. Thus it lifts to a ``fibrant approximation" functor on $\Mod(\cl)$ by sending $A$ to $RA$. Hence the result follows from \cite{S} B.2.
\end{proof}

\begin{rem}\label{re2.4}
{
\em
(1) The assumption that $\ct$ is \textit{locally cofibrant} (i.e. that it has all hom-objects cofibrant) is only needed for the existence 
of the projective model structure on $[\ct,\cv]$. Thus it can replaced by assuming that $\cv$ satisfes the monoid axiom.
 
(2) Each $\cv$ having all objects fibrant has $R=\Id$. In $\SSet$, one can take $R=\Ex^{\infty}$ because it is a colimit of a countable
chain of right adjoint functors (see \cite{GJ}) and filtered colimits commute with finite weighted limits in $\cv$ (see \cite{K}). 
Following \cite{J} B2.1.4. and B2.1.6, $R$ is a simplicial functor.

(3) We could replace a finite weighted limit sketch by an ($\alpha$-small) weighted limit sketch but we should assume that $R$ preserves
($\alpha$-small) weighted limits (see \cite{K}, 7.4).
}
\end{rem} 

\begin{defi}\label{def2.5}
{\em
A \textit{weighted homotopy limit sketch} is a weighted limit sketch $\ch=(\ct,L)$ where all weights $G_l$ are cofibrant in $[\ct,\cv]$. 
 
A \textit{homotopy model} of $\ch$ is a $\cv$-functor $A:\ct\to\cv$ such that the induced morphisms
$$
t_l^A:AX_l\to\{G,AD_l\}
$$
are weak equivalences for each $l\in L$. 
}
\end{defi}

Let us add that $\{G,AD_l\}$ is the weighted homotopy limit in the sense of \cite{V} provided that the diagrams $AD_l$ are pointwise cofibrant
for each $l\in L$, i.e., that all $AD_ld$, $d\in\cd$, $l\in L$ are cofibrant.

We will denote by $\HMod(\ch)$ the full subcategory of $[\ct,\cv]$ consisting of all homotopy models of $\ch$. Of course, 
any model of $\ch$ is a homotopy model of $\ch$. We say that a homotopy model is fibrant if it is fibrant in the projective model
structure on $[\ct,\cv]$.
 
\begin{exam}\label{ex2.6}
{
\em
Let $\ct$ be a small $\cv$-category and $f:X\to Y$ a morphism in $\ct$. Let $\cd$ be a free $\cv$-category over $1$. Thus $\cd$ 
has a unique object $d$ with $\cd(d,d)$ equal to the tensor unit $I$ of $\cv$. Let $\ch=(\ct,L)$ be a weighted limit sketch where $L$ 
consists of a single weight $G:\cd\to\cv$ with $Gd=I$, a single diagram $D:\cd\to\ct$ with $Dd=B$ and a morphism $\delta:I\to\ct(X,Y)$ 
corresponding to $f$. Models of $\ch$ are $\cv$-functors $A:\ct\to\cv$ such that $A(f)$ is an isomorphism. The weight $G$ is cofibrant 
and $A$ is a homotopy model of $\ch$ if and only if it is fibrant and $A(f)$ is a weak equivalence. 
}
\end{exam}

\begin{theo}\label{th2.7}
Let $\cv$ be a left proper tractable monoidal model category and $\ch=(\ct,L)$ a weighted homotopy limit sketch with $\ct$ locally cofibrant. 
Then there is a lo\-ca\-li\-zed model category structure $\cm_\ch$ on $[\ct,\cv]$ whose fibrant objects are precisely fibrant homotopy models of $\ch$.
\end{theo}
\begin{proof}
In $[\ct,\cv]$, we have morphisms
$$
\varphi_l:G_l\ast\ct(D_l,-)\to\ct(X_l,-)
$$
for each $l\in L$. Since hom-functors are always cofibrant in $[\ct,\cv]$, $G_l\ast\ct(D_l,-)$ is cofibrant as well (see \cite{LR} 4.1).
Thus $\varphi_l$ is a morphism between cofibrant objects for each $l\in L$. Following \cite{B}, 3.18., there exists a left Bousfield
$\cv$-localization $\cm_\ch$ of $[\ct,\cv]$ with respect to the set $\{\varphi_l\backslash l\in L\}$. Fibrant objects in $\cm_\ch$ are
fibrant objects in $[\ct,\cv]$ for which
$$
[\ct,\cv](\varphi_l,A):[\ct,\cv](\ct(X_l,-),A)\to[\ct,\cv](G_l\ast\ct(D_l,-),A)
$$
is a weak equivalence for each $l\in L$. Since 
$$
[\ct,\cv](\ct(X_l,-),A)\cong A(X_l)
$$ 
and 
$$
[\ct,\cv](G_l\ast\ct(D_l,-),A)\cong\{G_l,\ct(D_l,A)\},
$$ 
$[\ct,\cv](\varphi_l,A)$ corresponds to the morphism $t_l^A$. Thus a fibrant object $A$ in $[\ct,\cv]$ is fibrant in $\cm_\ch$ if and only if $A$ 
is a homotopy model of $\ch$.  
\end{proof}
 
\begin{lemma}\label{le2.8}
Let $\cv$ be a combinatorial monoidal model ca\-te\-go\-ry having all objects cofibrant and $\ch=(\ct,L)$ a weighted homotopy limit sketch. 
Let $G:\cd^{\op}\to\cv$ be a cofibrant weight, $D_1,D_2:\cd\to[\ct,\cv]$ diagrams and $\alpha:D_1\to D_2$ such that $\alpha_d$ is a weak equivalence
in $\cm_\ch$ for each object $d$ in $\cd$. Then $G\ast\alpha:G\ast D_1\to G\ast D_2$ is a weak equivalence in $\cm_\ch$. 
\end{lemma}
\begin{proof}
Since $[\ct,\cv]$ is a model $\cv$-category, there is a cofibrant replacement $\cv$-functor $Q:[\ct,\cv]\to[\ct,\cv]$ (see \cite{Sh} 24.2). Let
$\gamma_A:QA\to A$ denote the corresponding trivial fibration for $A\in[\ct,\cv]$. Consider the diagram
$$
\xymatrix@C=4pc@R=3pc{
G\ast QD_1 \ar[r]^{G\ast Q\alpha} \ar[d]_{G\ast\gamma D_1} &
G\ast QD_2\ar [d]^{G\ast\gamma D_2}\\
G\ast D_1 \ar[r]_{G\ast\alpha}& G\ast D_2
}
$$
For $D:\cd\to [\ct,\cv]$ and $X\in\ct$, let $D_X:\cd\to\cv$ be defined by $D_Xd=Dd(X)$. Since $(\gamma D)_X:D_X\to (QD)_X$ is an objectwise
weak equivalence between objectwise cofibrant diagrams, $G\ast(\gamma D)_X$ is a weak equivalence in $\cv$. This is \cite{H} 18.4.4, which is clearly
valid for model $\cv$-categories. Thus $G\ast\gamma D$ is a weak equivalence in $[\ct,\cv]$. Therefore $G\ast\alpha$ is a weak equivalence in $\cm_\ch$
if and only if $G\ast Q\alpha$ is a weak equivalence in $\cm_\ch$. Since a left Bousfield localization of a model $\cv$-category $[\ct,\cv]$
is a model $\cv$-category (see \cite{B}, 4.46) and $Q\alpha$ is an objectwise weak equivalence in $\cm_\ch$ between objectwise cofibrant diagrams,
$G\ast Q\alpha$ is a weak equivalence in $\cm_\ch$ (cf. \cite{H}, 18.4.4).
\end{proof}

\section{Homotopy limit theories}

\begin{defi}\label{def3.1}
{\em
A weighted limit sketch $\ch=(\ct,L)$ will be called \textit{normal} if $X_l=\{G_l,D_l\}$ and 
$$
\delta_l:G_l\to\ct(X_l,D_l)
$$
is the weighted limit cone.
}
\end{defi}

\begin{rem}\label{re3.2}
{
\em
(1) Let $\Phi$ be a class of finitely presentable cofibrant weights and $\ct$ be a small $\cv$-category having all limits weighted by weights belonging
to $\Phi$. We get a normal finite weighted homotopy limit sketch $\ch(\ct)=(\ct,L(\ct))$ where $L(\ct)$ consists of all pairs $(G,D)$ where 
$G:\cd\to\cv$ belongs to $\Phi$ and $D:\cd\to\ct$ is a diagram. These sketches will be called $\Phi$-\textit{weighted homotopy limit theories}. 
If $\Phi$ consists of all finitely presentable cofibrant weights, we say that $\ch(\ct)$ is a \textit{finite weighted homotopy limit theory}. 
Very often we will denote these theories just by $\ct$.  

For a pair of weights $H:\cc^{\op}\to\cv$ and $G:\cd^{\op}\to\cv$, we say that $H$-\textit{colimits commute with} $G$-\textit{limits} if the functor
$$
H\ast -:[\cc,\cv]\to\cv
$$
preserves limits weighted by $G$. The class of colimits commuting with all $\Phi$-weighted limits is denoted $\Phi^+$. Weights belonging to $\Phi^+$ 
are called $\Phi$-\textit{flat}.

(2) A weight $G:\cd^{\op}\to\cv$ will be called \textit{homotopy invariant} if, for its cofibrant replacement $\gamma_G:G_c\to G$ in the projective model 
structure and any objectwise cofibrant diagram $D:\cd\to\cv$ (i.e., $Dd$ is cofibrant for each object $d$ in $\cd$), the morphism  
$\gamma_G\ast D: G_c\ast D\to G\ast D$ is a weak equivalence.  
  
Following \cite{H} 18.4.5 (1) this definition does not depend on the choice of a cofibrant replacement. In particular, any cofibrant weight
is homotopy invariant. 

(3) Given a class $\Phi$ of cofibrant weights, then $\Phi^\diamond$ will denote the closure in cofibrant weights of $\Phi$ under colimits weighted 
by $\Phi$-flat homotopy invariant weights. This means that $\Phi^\diamond$ arises from $\Phi$ by iterative taking weighted colimits in presheaves
$G\ast D$ such that $G$ is $\Phi$-flat and homotopy invariant, $D$ is objectwise cofibrant and $G\ast D$ is cofibrant. 

Whenever $G$ is cofibrant and $D$ objectwise cofibrant then $G\ast D$ is cofibrant.

(4) Often, we will have to assume that all objects of $\cv$ are cofibrant. Then all diagrams $D:\cd\to\cv$ are objectwise cofibrant, 
which simplifies the definition of a homotopy invariant weight. Also, every $\cv$-category is locally cofibrant.
}
\end{rem}
 
\begin{theo}\label{th3.3}
Let $\cv$ be a finitely combinatorial monoidal model ca\-te\-go\-ry having all objects cofibrant and $\Phi$ a class of finitely presentable cofibrant 
weights such that every cofibrant weight is weakly equivalent to a weight belonging to $\Phi^\diamond$. Assume that $\cv$ is equipped with a fibrant 
approximation $\cv$-functor $R:\cv\to\cv$ preserving $\Phi$-weighted limits. Let $\ct$ be a $\Phi$-weighted homotopy limit theory. 
Then the model categories $\Mod(\ch(\ct))$ and $\cm_{\ch(\ct)}$ are Quillen equivalent.
\end{theo}
\begin{proof}
Let $\ch=\ch(\ct)$. The $\cv$-functor $F:[\ct,\cv]\to\Mod(\ch)$ is left Quillen (see the proof of \ref{th2.3}). Since all $\varphi_l$ from the proof 
of \ref{th2.7} are morphisms between cofibrant objects and $F\varphi_l$ are isomorphisms, $F:\cm_\ch\to [\ct,\cv]$ is a left Quillen functor. Thus $(U,F)$
is a Quillen pair between $\Mod(\ch)$ and $\cm_\ch$. We have to show that it is a Quillen equivalence of $\Mod(\ch)$ and $\cm_\ch$. Let $\eta:\Id\to UF$ 
be the unit of the adjunction. If $A$ is a cofibrant object in $[\ct,\cv]$ belonging to $\Phi$ then $A\cong A\ast Y$ where $Y:\ct^{\op}\to[\ct,\cv]$ 
is the Yoneda embedding. Since $A\in\Phi$, the pair $l=(A,\Id_\ct)$ belongs to $L(\ct)$. For $B$ in $\Mod(\ch)$ we have
\begin{align*}
\Mod(\ch)(Y(\{A,\Id_\ct\}),B)&\cong B(\{A,\Id_\ct\})\cong\{A,B\}\cong\cv(I,\{A,B\})\\
&\cong[\ct,\cv](A,\cv(I,B))\cong[\ct,\cv](A,B)\\
&\cong[\ct,\cv](A,\Mod(\ch)(Y,B))
\end{align*}
(the last isomorphism follows from the enriched Yoneda lemma, see \cite{B} 6.3.5). Thus $Y(\{A,\Id_\ct\})$ is the weighted colimit $A\ast Y$ 
in $\Mod(\ch)$. Consequently $\eta_A=\varphi_l$ and thus $\eta_A$ is a weak equivalence in $\cm_\ch$. 

Now, let $A$ be an arbitrary cofibrant object in $[\ct,\cv]$. Following our assumption, $A$ is weakly equivalent in $[\ct,\cv]$ to $A'$ belonging 
to $\Phi^\diamond$. Thus there is a zig-zag of weak equivalences in $[\ct,\cv]$ between $A$ and $A'$. Since both $A$ and $A'$ are cofibrant, 
this zig-zag can be changed into a zig-zag of weak equivalences in $[\ct,\cv]$ between cofibrant objects. Since $F$ is left Quillen and $U$ 
preserves weak equivalences, the composition $UF$ preserves weak equivalences between cofibrant objects. Thus $\eta_A$ is a weak equivalence 
in $\cm_\ch$ if and only if $\eta_{A'}$ is a weak equivalence in $\cm_\ch$. Hence, without any loss of generality, we can assume that $A\in\Phi^\diamond$. 
Consider a weighted colimit $G\ast D$ in $[\ct,\cv]$ where $G$ is $\Phi$-flat and homotopy invariant, $D:\cd\to\Phi^\diamond$, $G\ast D$ is cofibrant
and $\eta_{Dd}$ is a weak equivalence in $\cm_\ch$ for each $d\in\cd$. We have to prove that $\eta_{G\ast D}$ is a weak equivalence in $\cm_\ch$. 
Since $\Phi^+$-weighted colimits commute with $\Phi$-weighted limits, the functor $U$ preserves $\Phi^+$-weighted colimits. Thus the unit 
$\eta_{G\ast D}:G\ast D\to UF(G\ast D)$ is a $\Phi^+$-weighted colimit $G\ast\eta_{Dd}$ of units $\eta_{Dd}$ where $Dd\in\Phi$. Consider the commutative 
diagram
$$
\xymatrix@C=4pc@R=3pc{
G_c\ast D \ar[r]^{\gamma_G\ast D} \ar[d]_{\eta_{G_c\ast D}} &
G\ast D\ar [d]^{\eta_{G\ast D}}\\
UF(G_c\ast D) \ar[r]_{UF(\gamma_G\ast D)}& UF(G\ast D)
}
$$
For $X\in\ct$, let $D_X:\cd\to\cv$ be defined by $D_Xd=Dd(X)$. Since $G$ is homotopy invariant and $D_X$ is objectwise cofibrant, 
$(\gamma_G\ast D)_X = \gamma_G\ast D_X$ is a weak equivalence in $\cv$ for each $X\in\ct$. Thus $\gamma_G\ast D$ is a weak equivalence in $[\ct,\cv]$. 
Since $UF$ preserves weak equivalences between cofibrant objects and both $G_c\ast D$ and $G\ast D$ are cofibrant, $UF(\gamma_G\ast D)$ is a weak 
equivalence in $[\ct,\cv]$. Thus it suffices to prove that $\eta_{G_c\ast D}$ is a weak equivalence in $\cm_\ch$. 

Following \ref{le2.8},  $G_c\ast\eta D:G_c\ast D \to G_c\ast UFD$ is a weak equivalence in $\cm_\ch$. Since $\eta_{G_c\ast D}$ is the composition 
$k(G_c\ast\eta D)$ where $k: G_c\ast UFD \to UF(G_c\ast D)$ is the induced morphism, it remains to prove that $k$ is a weak equivalence in $[\ct,\cv]$. 
We have the commutative square
$$
\xymatrix@C=4pc@R=3pc{
G_c\ast UFD \ar[r]^{\gamma_G\ast UFD} \ar[d]_{k} &
G\ast UFD\ar [d]^{k'}\\
UF(G_c\ast D) \ar[r]_{UF(\gamma_G\ast D)}& UF(G\ast D)
}
$$
where $k'$ is the induced isomorphism. Since the diagram $(UFD)_X$ is objectvise cofibrant for each $X\in\ct$, 
$(\gamma_G\ast UFD)_X=\gamma_G\ast (UFD)_X$ is a weak equivalence in $\cv$. Thus $(\gamma_G\ast UFD)$ is a weak equivalence in $[\ct,\cv]$.
Since we have shown that $UF(\gamma_G\ast D)$ is a weak equivalence in $[\ct,\cv]$, $k$ is a weak equivalence in $[\ct,\cv]$ as well.

Let $f:A\to B$ be a mor\-phism between fibrant objects in $\Mod(\ch)$ such that $Uf$ is a weak equivalence. Then $Uf$ is a weak equivalence 
between fibrant objects in $\cm_\ch$ and thus it is a weak equivalence in $[\ct,\cv]$ (see \cite{H} 3.2.13). Hence $f$ is a weak equivalence, 
which means that $U$ reflects weak equivalences between fibrant objects. Finally, since $[\ct,\cv]$ and $\cm_\ch$ have the same cofibrant objects
and $U$ preserves weak equivalences, 
$$
\overline{\eta}_A: A \xrightarrow{\eta_A} UFA \xrightarrow{U\rho_{FA}} URFA
$$
is a weak equivalence for each cofibrant object $A$ in $\cm_\ch$. Following \cite{Ho} 1.3.16, $(U,F)$ is a Quillen equivalence.
\end{proof}
 
\begin{theo}\label{th3.4}
Let $\cv$ be a left proper finitely tractable monoidal model ca\-te\-go\-ry satisfying the monoid axiom and equipped with a fibrant approximation 
$\cv$-functor $R:\cv\to\cv$ preserving finite weighted limits. Let $\ct$ be a locally cofibrant finite weighted homotopy limit theory. 
Then the model categories $\Mod(\ch(\ct))$ and $\cm_{\ch(\ct)}$ are Quillen equivalent.
\end{theo}
\begin{proof}
Let $\Phi$ be the class of all finitely presentable cofibrant weights. Following \cite{MRV} 5.1, any cofibrant object $A$ in $[\ct,\cv]$ is a colimit 
of a directed diagram $D:\cd\to [\ct,\cv]$ such that $Dd$ is cofibrant and finitely presentable in $[\ct,\cv]$ for each $d\in\cd$. This directed colimit 
is a weighted colimit $\overline{\Delta I}\ast\overline{D}$ where $\overline{\cd}$ is a free $\cv$-category on $\cd$, $\overline{D}:\overline{\cd}\to\ [\ct,\cv]$
is the extension of $D$ and $\overline{\Delta I}$ is the extension of the constant diagram $\Delta I:\cd\to\cv$ on $I$. This weighted colimit
is $\Phi$-flat (see \cite{K}, 4.9) and homotopy invariant (see \cite{LR} 4.5).  Following the proof of \ref{th3.3}, $\eta_{Dd}$ is a weak equivalence
in $\cm_\ch$ between cofibrant objects for each $d\in\cd$. This proof also yields that $\eta_A$ is a weak equivalence in $\cm_\ch$ and thus it proves
the result. We do not need \ref{le2.8} because $\eta D$ is an objectwise weak equivalence in $\cm_\ch$ between objectwise cofibrant diagrams and thus 
$G_c\ast\eta D$ is a weak equivalence in $\cm_\ch$ (cf. \cite{H} 18.4.4). Since each cofibration in $[\ct,\cv]$ is an objectwise cofibration
(see \cite{Sh} 24.4), the diagrams $D_X$ and $(UFD)_X$ are objectwise cofibrant for each $X\in\ct$. Thus, in the whole proof, we do not need to assume
that all objects of $\cv$ are cofibrant.
\end{proof}

\begin{rem}\label{re3.5}
{
\em
 
(1) Recall that an \textit{algebraic} $\cv$-\textit{theory} is a small $\cv$-category $\ct$ with finite products and a $\ct$-\textit{algebra} 
is a $\cv$-functor $A:\ct\to\cv$ preserving finite products. A \textit{homotopy} $\ct$-\textit{algebra} preserves finite products
up to a weak equivalence, i.e., 
$$
A(X_1\times\dots\times X_n)\to A(X_1)\times\dots\times A(X_n)
$$
are weak equivalences.  

Let $\Phi$ consist of constant weights on finite discrete ca\-te\-go\-ries with the value $I$. Then $\Phi$-weighted homotopy limit theories are precisely
algebraic theories. Following \cite{KS} 3.8, the saturation $\Phi^\ast$ consists of finite coproducts of representables. It is easy to see that
the corresponding sets of morphisms $\varphi_l$ for $\Phi$ and $\Phi^\ast$ are equal for each algebraic theory $\ct$. Thus both algebras and
homotopy algebras are unchanged by the passage from $\Phi$ to $\Phi^\ast$. Over simplicial sets, every homotopy colimit is weakly equivalent 
to a homotopy invariant $\Phi^{\ast +}$-colimit of finite coproducts of representables (see \cite{V}). Thus the result of Badzioch and Bergner 
is a consequence of \ref{th3.3} applied to $\Phi^\ast$. We can not use $\Phi$ for this purpose because $[\ct,\cv]$ contains no elements of $\Phi$, 
and so $\Phi^\diamond$ is also empty. Observe that the saturation does not change the flatness, i.e., $\Phi^+=\Phi^{\ast +}$. 

J. Bourke \cite{Bo} proved that, over $\Cat$, any cofibrant weight belongs to the iterative closure of finite coproducts of representables under
co\-li\-mits weighted by homotopy invariant $\Phi$-flat weights. Thus there is a rigidification theorem for homotopy algebras in this case as well.

(2) Let $\Phi$ consist of constant weights on finite discrete ca\-te\-go\-ries with the value $I$ and of weights on the single object discrete 
category with a finitely presentable cofibrant value. Then $\Phi$-weighted homotopy limit theories contain finite products and cotensors with finitely 
presentable cofibrant objects and are related to enriched Lawvere theories in the sense of \cite{P}. Over $\SSet$, \cite{V} and \ref{th3.3} yield 
the rigidification theorem for them. Again, we have pass to the saturation $\Phi^\ast$ of $\Phi$.
}
\end{rem}

\section{Conservative free completion}

\begin{defi}\label{def4.1}
{\em
Let $\ch=(\ct,L)$ be a finite weighted homotopy limit sketch. We say that $E:\ct\to\ct^\ast$ is a \textit{conservative free completion} of $\ch$ 
if $\ct^\ast$ has limits weighted by finitely presentable cofibrant weights of diagrams $D:\cd\to\ct$ and, for each model $A:\ct\to\cv$, there is 
an \textit{essentially unique} (i.e., unique up to an isomorphism) $\cv$-functor $A^\ast:\ct^\ast\to\cv$ preserving limits weighted by finitely 
presentable cofibrant weights of diagrams $D:\cd\to\ct$ such that $A^\ast E\cong A$. 
}
\end{defi}

\begin{lemma}\label{le4.2}
Let $\cv$ be locally finitely presentable as a closed category. Then each finite weighted homotopy limit sketch has a conservative free completion. 
Moreover, the functor $E:\ct\to\ct^\ast$ is a full embedding provided that $\ch$ is normal.
\end{lemma}
\begin{proof}
Let $\ct$ be a finite weighted homotopy limit sketch. Then 
$$
Y:\ct^{\op}\to[\ct,\cv]
$$ 
is a free completion of $\ct^{\op}$ under weighted colimits and the full subcategory $[\ct,\cv]_{\fp}$ of $[\ct,\cv]$ consisting of finitely presentable 
objects is a free completion of $\ct^{\op}$ under finite weighted colimits (see \cite{KS}, 3.13). This implies that for each $\cv$-functor 
$H:\ct^{\op}\to\cv^{\op}$, there is an essentially unique $\cv$-functor $\overline{H}:[\ct,\cv]\to\cv^{\op}$ preserving weighted colimits such that 
$\overline{H}Y\cong H$. Moreover, its restriction $\overline{H}_0:[\ct,\cv]_{\fp}\to\cv^{\op}$ is an essentially unique $\cv$-functor preserving finite 
weighted colimits such that $\overline{H}_0Y\cong H$. Let $G:\cd^{\op}\to\cv$ be a finitely presentable weight and $D:\cd\to[\ct,\cv]_{\fp}$ a diagram. 
Following \cite{KS}, 3.13 again, $G$ belongs to the closure of representable functors under finite colimits in $[\cd^{\op},\cv]$. The $\cv$-functor 
$$
-\ast D:[\cd^{\op},\cv]\to[\ct,\cv]
$$ 
preserves weighted colimits because it has a right $\cv$-adjoint 
$$
[\cd^{\op},\cv](-,[\ct,\cv](D,-)):[\ct,\cv]\to[\cd^{\op},\cv].
$$ 
Since $\cd(-,d)\ast D=Dd$, $[\ct,\cv]_{\fp}$ has colimits weighted by finitely presentable weights and $\overline{H}_0:[\ct,\cv]_{\fp}\to\cv^{\op}$  
preserves them. Let $[\ct,\cv]_{\cfp}$ be the full subcategory of $[\ct,\cv]$ consisting of finitely presentable cofibrant objects. Since the class 
of cofibrant weights is saturated (see \cite{LR}), $[\ct,\cv]_{\cfp}$ is closed under colimits weighted by finitely presentable cofibrant weights. 
Moreover, $\overline{H}_1:[\ct,\cv]_{\cfp}\to\cv^{\op}$ preserves these co\-li\-mits.  

Following \cite{R} 4.6, $[\ct,\cv]$ is a finitely tractable model category with cofibrations consisting of all morphisms and weak equivalences 
of isomorphisms. The same is true for $\cv$ and $[\ct,\cv]$ is a model $\cv$-category where both $[\ct,\cv]$ and $\cv$ are taken 
with this trivial model structure. Following \cite{B} 3.18., there exists a left Bousfield $\cv$-localization $\overline{\cm}$ of $[\ct,\cv]$ 
with respect to the set $\{\varphi_l\backslash l\in L\}$ from the proof of \ref{th2.7}. Then $\overline{H}:[\ct,\cv]\to\cv^{\op}$ is a left Quillen 
functor where $\cv^{\op}$ is again taken with the trivial model structure. Since $\overline{H}$ sends each $\varphi_l$ to an isomorphism and $\Id$ 
is a cofibrant replacement functor on $[\ct,\cv]$, $\overline{H}:\overline{\cm}\to\cv^{\op}$ is a left Quillen functor (see \cite{H} 3.3.18 and \cite{B}). 
Thus it induces a left adjoint functor $H':\Ho\overline{\cm}\to\Ho\cv^{\op}$ (see \cite{Ho} 1.3.10). Following \cite{R} 4.6, $\Ho\overline{\cm}$ 
is the reflective full subcategory of $[\ct,\cv]$ consisting of objects orthogonal to each $\varphi_l$ and the canonical functor 
$P:\overline{\cm}\to\Ho\overline{\cm}$ is the corresponding reflector. Since $\cv^{\op}=\Ho\cv^{\op}$, both $\Ho\overline{\cm}$ and 
$\Ho\cv^{\op}$ are $\cv$-categories. Let $J:\Ho\overline{\cm}\to\overline{\cm}$ denote the inclusion $\cv$-functor. Since 
$H'\cong H'PJ\cong \overline{H}J$, $H'$ is a $\cv$-functor. We have
$$
H'(V\cdot PA)\cong H'P(V\cdot A)\cong\overline{H}(V\cdot A)\cong V\cdot\overline{H}A\cong V\cdot H'PA.
$$
Since $P$ is surjective on objects, $H'$ preserves tensors. Thus $H'$ is a left $\cv$-adjoint (see \cite{B} 6.7.6) and, consequently, it preserves
weighted co\-li\-mits. Thus the full subcategory $P([\ct,\cv]_{\cfp})^{op}$ of $(\Ho\overline{\cm})^{\op}$ consisting of $P$-images of objects from 
$[\ct,\cv]_{\cfp}$ is a conservative free completion of $\ch$.  

It follows from the construction of $E:\ct\to\ct^\ast$ as $(PY)^{\op}$ that morphisms $E_{AB}:\ct(A,B)\to\ct^\ast(EA,EB)$ are split epimorphisms
(as the composition of isomorphisms $Y^{\op}_{AB}$ with morphisms $P^{\op}_{YA,YB}$ split by $J^{\op}_{YA,YB}$). The Yoneda embedding 
$\overline{Y}:\ct\to[\ct^{\op},\cv]$ preserves existing weighted limits provided that $\ch$ is normal. Since $\overline{Y}\cong (\overline{Y})^\ast E$, 
the $E_{AB}$ are monomorphisms. Thus they are isomorphisms, which proves that $E$ is a full embedding.
\end{proof}
 
\begin{rem}\label{re4.3}
{
\em
(1) Since I could not find any reference for the existence of a conservative $\cv$-completion, I gave the proof above written in the language
of model categories. 

(2) Let $\cv$ be a finitely combinatorial monoidal model category having all objects cofibrant and equipped with a fibrant approximation 
$\cv$-functor $R:\cv\to\cv$ preserving finite weighted limits. We need this assumption because $\ct$ being locally cofibrant does not imply 
that $\ct^\ast$ is locally cofibrant. Consider the commutative square
$$
\xymatrix@C=4pc@R=3pc{
\Mod(\ch^\ast) \ar[r]^{U_{\ct^\ast}} \ar[d]_{U_1} &
\cm_{\ch^\ast}\ar [d]^{U_2}\\
\Mod(\ch) \ar[r]_{U_\ct}& \cm_\ch
}
$$
where both $U_1$ and $U_2$ are given by the precomposition with $E$. The functor $U_1$ is an equivalence of categories.

If $B:\ct^\ast\to\cv$ is a homotopy model of $\ch^\ast$ then $U_2B=BE$ is a homotopy model of $\ch$. But we can prove more.
} 
\end{rem}

\begin{lemma}\label{le4.4}
$U_2$ is a right Quillen functor.
\end{lemma}
\begin{proof}
We will prove that the left adjoint $F_2$ of $U_2$ is left Quillen. Consider a morphism $\varphi_l:G_l\ast\ct(D_l,-)\to\ct(X_l,-)$
from the proof of \ref{th2.7}. Since $F_2\varphi_l:G_l\ast\ct^\ast(ED_l,-)\to\ct^\ast(EX_l,-)$ is the corresponding morphism in $\ct^\ast$,
it is a weak equivalence in $\cm_{\ch^\ast}$. Let $Q$ be a cofibrant replacement functor on $\cm_{\ch}$. We have a commutative square
$$
\xymatrix@C=4pc@R=3pc{
F_2QA \ar[r]^{F_2Q\varphi} \ar[d]_{F_2\gamma_A} &
F_2QB\ar [d]^{F_2\gamma_B}\\
F_2A \ar[r]_{F_2\varphi}& F_2B
}
$$
where $\varphi:A\to B$ denotes $\varphi_l$. The morphisms $\gamma_A$ and $\gamma_B$ are trivial fibrations in $[\ct,\cv]$ and 
$F_2:[\ct,\cv]\to [\ct^\ast,\cv]$ is left Quillen. Since $\varphi$ is a morphism between cofibrant objects, $F_2\gamma_A$ and $F_2\gamma_B$ are 
weak equivalences in $[\ct^\ast,\cv]$ and thus in $\cm_{\ch^\ast}$. Hence $F_2Q\varphi$ is a weak equivalence. Following \cite{H} 3.3.18, 
$F_2:\cm_\ch\to\cm_{\ch^\ast}$ is left Quillen.
\end{proof}

\begin{rem}\label{re4.5}
{
\em
The functor $U_2$ has a right $\cv$-adjoint $S$ as well. If $\ct$ is normal, this right adjoint can be easily calculated.

Following the proof of \ref{le4.2}, $\ct^\ast$ is a coreflective full subcategory of $\tilde{\ct}=([\ct,\cv]_{\cfp})^{\op}$ with the inclusion
$\overline{J}=J^{\op}$ and its right $\cv$-adjoint $\overline{P}=P^{\op}$. We use the notation of the proof of \ref{le4.2}, except that $J$
and $P$ are the domain restrictions of those from \ref{le4.2}. The counit of this adjunction will be denoted 
$\varepsilon:\overline{J}\overline{P}\to\Id$. Then $E:\ct\to\ct^\ast$ is the composition $\overline{P}\overline{Y}$ where
$\overline{Y}:\ct\to\tilde{\ct}$ is the dual of the codomain restriction of the Yoneda embedding $\ct^{\op}\to[\ct,\cv]$. Since the values 
of $\overline{Y}$ belong to $\ct^\ast$, we have $\overline{J}\overline{P}\overline{Y}\cong\overline{Y}$.

The $\cv$-category $[\ct,\cv]$ is isomorphic to a full reflective subcategory of $[\tilde{\ct},\cv]$ consisting of $\cv$ functors preserving weighted
limits. Thus the restriction $\cv$-functor $[\overline{Y},\cv]$ has a right $\cv$-adjoint sending $A:\ct\to\cv$ to its weighted limit preserving
extension $\tilde{A}:\tilde{\ct}\to\cv$. The $\cv$-functor $[\overline{P},\cv]:[\ct^\ast,\cv]$ has a right $\cv$-adjoint $[\overline{J},\cv]$. 
Thus $U_2=[E,\cv]=[\overline{P}\overline{Y},\cv]$ has a right $\cv$-adjoint $\tilde{-}\overline{J}$ sending $A$ to the composition
$\tilde{A}\overline{J}$.
}
\end{rem}

\begin{lemma}\label{le4.6}
$U_2$ preserves cofibrations.
\end{lemma}
\begin{proof}
The claim is equivalent to the preservation of trivial fibrations by the right $\cv$-adjoint $S$. Consider a trivial fibration $\alpha:A\to B$
and an object $\{G,D\}$ in $\tilde{\ct}$ where $G$ is a finitely presentable cofibrant weight. Since $\tilde{\alpha}_{\{G,D\}}=\{G,\alpha D\}$, 
we get (analogously as in \cite{H} 18.4.2) that $\tilde{\alpha}$ is a trivial fibration. Thus $S\alpha=\tilde{\alpha}\overline{J}$ is a trivial
fibration.
\end{proof}

\section{Rigidification}

\begin{theo}\label{th5.1} 
Let $\cv$ be a finitely combinatorial monoidal model ca\-te\-go\-ry having all objects cofibrant and equipped with a fibrant appro\-xi\-ma\-tion 
$\cv$-functor $R:\cv\to\cv$ preserving finite weighted limits. Let $\ch=(\ct,L)$ be a normal finite weighted homotopy limit sketch. Then the following 
conditions are equivalent:
\begin{enumerate}
\item[(i)] $(U_\ct,F_\ct)$ is a Quillen equivalence,
\item[(ii)] each homotopy model of $\ch$ is weakly equivalent in $[\ct,\cv]$ to a model of $\ch$,
\item[(iii)] $(U_2,F_2)$ is a Quillen equivalence, and
\item[(iv)] each fibrant homotopy model $A$ of $\ch$ is weakly equivalent in $[\ct,\cv]$ to $U_2B$ for a fibrant homotopy model $B$ of $\ch^\ast$.
\end{enumerate}
\end{theo}
\begin{proof}
Since the square from \ref{re4.3} (2) consists of right Quillen functors, $U_1$ is an equivalence and $U_{\ct^\ast}$ is a Quillen equivalence 
(see \ref{th3.4}), \cite{Ho} 1.3.15 implies that $(i)\Leftrightarrow (iii)$.
  
$(i)\Rightarrow (ii)$:  Let $A$ be a homotopy model of $\ch$. Since $R$ preserves finite weighted limits, $RA$ is a fibrant homotopy model of $\ch$.
Moreover, the natural transformation $\rho_A:A\to RA$ is a pointwise trivial cofibration and thus a weak equivalence in $[\ct,\cv]$. Now, we take 
a cofibrant replacement $\gamma_{RA}:QRA\to RA$ of $RA$ in $[\ct,\cv]$. Consider a commutative square 
$$
\xymatrix@C=4pc@R=3pc{
QRA\{G_l,D_l\} \ar[r]^{t^{QRA}_l} \ar[d]_{\gamma_{RA\{G_l,D_l\}}} &
\{G_l,QRAD_l\}\ar [d]^{\{G_l,\gamma_{RAD_l}\}}\\
RA\{G_l,D_l\} \ar[r]_{t^{RA}_l}& \{G_l,RAD_l\}
}
$$
Since $\gamma_{RAD_l}$ is a weak equivalence between fibrant objects in $[\ct,\cv]$, the right vertical morphism $\{G_l,\gamma_{RAD_l}\}$
is a weak equivalence (analogously as in \cite{H} 18.4.4). Hence $t^{QRA}_l$ is a weak equivalence and thus $QRA$ is a homotopy model
of $\ch$. Since $\gamma_{RA}$ is a trivial fibration between fibrant objects in $\cm_\ch$, it is a weak equivalence in $[\ct,\cv]$. Thus $A$ 
is weakly equivalent in $[\ct,\cv]$ to the homotopy model $QRA$ of $\ch$ which is fibrant and cofibrant in $[\ct,\cv]$. Therefore, it suffices
to prove $(ii)$ for any homotopy model $A$ of $\ch$ which is fibrant and cofibrant in $[\ct,\cv]$.
Following \cite{Ho} 1.3.13, we have a weak equivalence 
$$
\overline{\eta}_A: A \xrightarrow{\eta_A} U_{\ct}F_{\ct}A \xrightarrow{U_{\ct}\rho_{F_{\ct}A}} U_{\ct}RF_{\ct}A
$$
Since $A$ is fibrant in $\cm_\ch$ (see \ref{th2.7}), $\overline{\eta}_A$ is a weak equivalence between fibrant objects in $\cm_\ch$. Thus 
$\overline{\eta}_A$ is a weak equivalence in $[\ct,\cv]$ (see \cite{H} 3.2.13). Since $RF_{\ct}A$ is a model of $\ch$, $(ii)$ is proved. 
 
$(iv)\Rightarrow(iii)$: At first, we will show that $U_2$ reflects weak equivalences between fibrant objects. Let $f:A\to B$ be a morphism between 
fibrant objects in $\cm_{\ch^\ast}$ such that $U_2f$ is a weak equivalence in $\cm_\ch$. Since $U_2f$ is a weak equivalence between fibrant objects, 
it is a weak equivalence in $[\ct,\cv]$. Consider an object $\{G,ED\}\in\ct^\ast$ which does not belong to the image of $E$. Then $fED$ is a pointwise 
weak equivalence between fibrant objects. Since $G$ is cofibrant, $\{G,EDf\}$ is a weak equivalence (cf. \cite{H} 18.4.4). Consequently, $f$ is a weak 
equivalence.

Since $E$ is a full embedding (see \ref{le4.2}), the adjunction units $\eta_X:X\to G_2F_2X$ are isomorphisms. Following \cite{Ho} 1.3.16, $(U_2,F_2)$ 
is a Quillen equivalence if and only if $U_2\rho_{F_2X}$ is a weak equivalence for each cofibrant object $X$. We will prove at first that $U_2\varphi_l$ 
is a weak equivalence for each $\varphi_l$ from $\ch^\ast$. We know that $[\ct^\ast,\cv](\varphi_l,Z)$ is a weak equivalence in $\cv$ for each fibrant 
object $Z$ of $\cm_{\ch^\ast}$. Let $Z_1$ be a fibrant object in $\cm_\ch$. Assuming $(iv)$, there is a fibrant object $Z$ in $\cm_{\ch^\ast}$ such
that $QZ_1$ is weakly equivalent in $[\ct,\cv]$ to $U_2Z$. Consider the cofibrant replacement $\gamma^\ast_Z:Q^\ast Z\to Z$ of $Z$. Since $\gamma^\ast_Z$
is a trivial fibration in $\cm_{\ch^\ast}$, it is a weak equivalence in $[\ct,\cv]$ and $Q^\ast Z$ is a fibrant homotopy model of $\ch^\ast$ because
it is a fibrant object in $\cm_{\ch^\ast}$. Thus $QZ_1$ is weakly equivalent to $U_2Q^\ast Z$ in $[\ct,\cv]$. Since $U_2$ preserves cofibrations 
(see \ref{le4.6}), both $QZ_1$ and $U_2Q^\ast Z$ are fibrant and cofibrant in $[\ct,\cv]$. Thus they are homotopy equivalent in $[\ct,\cv]$. Thus
there exists a homotopy equivalence $h:U_2Q^\ast Z\to QZ_1$. The composition $g=\gamma_{Z_1}h:U_2Q^\ast Z\to Z_1$ is a weak equivalence between
fibrant objects in $[\ct,\cv]$. Consider the commutative square 
$$
\xymatrix@C=7pc@R=4pc{
[\ct,\cv](\cod U_2\varphi_l,U_2Q^\ast Z) \ar[r]^{[\ct,\cv](U_2\varphi_l,U_2Q^\ast Z)} \ar[d]_{[\ct,\cv](\cod U_2\varphi_l,g)} &
[\ct,\cv](\dom U_2\varphi_l,U_2Q^\ast Z) \ar [d]^{[\ct,\cv](\dom U_2\varphi_l,g)}\\
[\ct,\cv](\cod U_2\varphi_l,Z_1)  \ar[r]_{[\ct,\cv](U_2\varphi_l,Z_1)}& [\ct,\cv](\dom U_2\varphi_l,Z_1)  
}
$$
Since $\varphi_l$ has the domain and the codomain cofibrant and $U_2$ preserves cofibrations, the vertical morphisms are weak equivalences
because $g$ is a weak equivalence between fibrant objects. For the same reason, the upper horizontal morphism is a weak equivalence. Thus
$[\ct,\cv](U_2\varphi_l,Z_1)$ is a weak equivalence. Since $U_2\varphi_l$ has the domain and the codomain cofibrant, it is a weak equivalence 
in $\cm_\ch$.  

Take the (cofibration, trivial fibration) factorization  
$$
\varphi_l:A_l \xrightarrow{\quad  \varphi^l_l\quad} C_l \xrightarrow{\quad \varphi^2_l \quad} B_l.
$$
Since $U_2$ preserves projective weak equivalences and cofibrations and $U_2\varphi_l$ is a weak equivalence, $U_2\varphi^1_l$ is a trivial cofibration.
Following the properties of enriched left Bousfield localizations (see \cite{H} and \cite{B}), a fibrant object $Z$ in $[\ct^\ast,\cv]$ is fibrant
in $\cm_{\ch^\ast}$ if and only if it is injective to horns $\varphi^1_l\boxdot i$ where $i:V\to V'$ is a generating cofibration in $\cv$. Recall 
that $\varphi^1_l\boxdot i:K\to V'\cdot C_l$ is the pushout corner morphism for the pushout
$$
\xymatrix@C=4pc@R=3pc{
V\cdot A_l \ar[r]^{V\cdot\varphi^1_l} \ar[d]_{i\cdot A_l} &
V\cdot C_l\ar [d]^{}\\
V'\cdot A_l \ar[r]_{}& K
}
$$
Thus $\rho_{F_2X}$ is cofibrantly generated by these horns and trivial cofibrations in $[\ct^\ast,\cv]$. Since horns are trivial cofibrations 
in $\cm_{\ch^\ast}$ and $U_2$ preserves colimits and cofibrations, $U_2(\varphi^1_l\boxdot i)$ is a trivial cofibration. Thus $U_2\rho_{F_2X}$ 
is a weak equivalence.
 
$(ii)\Rightarrow(iv)$: Let $A$ be a fibrant homotopy model of $\ch$. Following $(ii)$, $A$ is weakly equivalent in $[\ct,\cv]$ to 
$U_\ct B_1=U_\ct U_1B_2$. Like at the beginning of the proof of $(i)\Rightarrow (ii)$,  $B_2$ is weakly equivalent in $[\ct^\ast,\cv]$
to the fibrant $\ch^\ast$-model $R^\ast B_2$. Thus $A$ is weakly equivalent in $[\ct,\cv]$ to $U_2U_{\ct^\ast}R^\ast B_2$ 
and $U_{\ct^\ast}R^\ast B_2$ is a fibrant homotopy model of $\ch^\ast$.
\end{proof} 

\begin{rem}\label{re5.2}
{
\em
(1) In the proof above (the implication $(i)\Rightarrow (ii)$), we showed that any fibrant $\cv$-functor $\ct\to\cv$ weakly equivalent
in $[\ct,\cv]$ to a fibrant homotopy model of $\ch$ is a homotopy model of $\ch$. In order that homotopy models of $\ch$ are closed
under weak equivalences in $[\ct,\cv]$, we need that $\{G_l,-\}$ preserves weak equivalences for each $l\in L$.

(2) We also showed that, assuming (iv), $U_2\varphi_l$ is a weak equivalence in $\cm_\ch$. Thus $[\ct,\cv](U_2\varphi_l,A)$
is a weak equivalence for each fibrant homotopy model $A$ of $\ch$. Hence $[\ct^\ast,\cv](\varphi_l,S(A))$ is a weak equivalence. Therefore
$S(A)$ is a fibrant homotopy model of $\ch^\ast$. Since $U_2S(A)=A$, (iv) is equivalent to

(v) $S$ preserves fibrant objects.

(3) Without any change of the proof, \ref{th5.1} is valid for any class $\Phi$ of finitely presentable cofibrant weights from \ref{th3.3}.
In particular we can take $\Phi^\ast$ from \ref{re3.5}. In this case, we get a characterization when any homotopy algebra of a normal finite
product sketch is equivalent to a strict algebra. The following example shows that this is not always true and indicates that this fact
is a kind of a coherence theorem.
}
\end{rem} 
 
\begin{exam}\label{ex5.3}
{
\em
Let $\ch=(\ct,L)$ be a normal finite product sketch for monoids. This means that $\ct$ contains objects $X_0,X_1,X_2,X_3$ where $X_0$ is terminal,
$X_2=X_1\times X_1$ and $X_3=X_1\times X_1\times X_1$. Moreover, $\ct$ contains morphisms $e:X_0\to X_1$ and $m:X_2\to X_1$ playing the role
of unit and multiplication subjected to the axioms for unit and the associativity axiom. Models of $\ch$ in $\Cat$ are precisely strict monoidal
categories. Assume that each homotopy model of $\ch$ is weakly equivalent in $[\ct,\Cat]$ to a model of $\ch$. Since equivalences of categories
are closed under finite products, \ref{re5.2} (1) implies that homotopy models are precisely functors $\ct\to\Cat$ equivalent to models of $\ch$.
This means that any homotopy model is equivalent by a strong monoidal functor to a strict monoidal category. Hence homotopy models of $\ch$ 
are precisely monoidal categories (see \cite{ML} XI.1-3.). But this is not possible because the pentagon axiom does not follow from the associativity 
up to a natural isomorphism. 

On the other hand, by adding an object $X_4= X_1^4$ to $\ct$, each homotopy model of the enlarged sketch is weakly equivalent to a model of $\ch$.
Of course, this example leans on \cite{Bo} as mentioned in \ref{re3.5}.
}
\end{exam}

\begin{exam}\label{ex5.4}
{
\em
(1) Recall that a \textit{Segal category} is a bisimplicial set $A:\boldsymbol{\Delta}^{\op}\to\SSet$ such that $A_0$ is a discrete
simplicial set and the Segal maps 
$$
s_k:A_k\to A_1\times_{A_0}\dots\times_{A_0} A_1
$$
(on the right side we have a limit of $k$ copies of $A_1$ over $A_0$) are weak equivalences for each $k\geq 2$ (see \cite{HS}). 
Thus Segal categories are precisely homotopy models $A$ of the finite limit theory of categories in $\SSet$ such that $A_0$
is discrete. When we fix the set $A_0$, this finite limit theory turns into an algebraic theory and thus each Segal category is equivalent
to a simplicial category, i.e., to a category enriched over $\SSet$ (see \cite{Be1}).

We can also sketch discreteness by forcing $A_0$ to be the cotensor $\Delta_1\pitchfork A_0$. Models of the resulting finite normal weighted limit 
sketch $\ch$ are precisely simplicial categories. This sketch is not a weighted homotopy limit sketch because we use limits (multiple pullbacks) 
and not homotopy limits. Let $\ch'$ be a weighted homotopy limit sketch associated to $\ch$. This means that each weight $G_l$ is substituted 
by its cofibrant replacement $G'_l$. Following \cite{H} 18.1.8, given a diagram $D:\cd\to\boldsymbol{\Delta}^{\op}$ this replacement is 
$B(\cd\downarrow -)$. For a multiple pullback diagram, this weight is finitely presentable. Since the weight for the cotensor
$\Delta_1\pitchfork A_0$ is $\Delta_1$ (as the constant functor from a single morphism category to $\SSet$), it is cofibrant. Homotopy models
$A$ of $\ch'$ can be called weak Segal categories because discreteness of $A_0$ is replaced by $A_0\to\Delta_1\pitchfork A_0$ being a weak 
equivalence. This means that $A_0$ is homotopy discrete, i.e., a coproduct of contractible simplicial sets. 

Now, let $A$ be a fibrant Segal category. Since $A_1$ is fibrant and $A_0$ is discrete, morphisms $A_1\to A_0$ are fibrations. Since homotopy pullbacks 
are isomorphic with pullbacks in this case, $A$ is a homotopy model of $\ch'$. Since $\Ex^\infty$ preserves discrete simplicial sets (see \cite{GJ} 
4.2), any Segal category $A$ is weakly equivalent to a fibrant Segal category $\Ex^\infty A$. Hence $A$ is weakly equivalent to a simplicial 
category. But I do not know whether each homotopy model of $\ch'$ is weakly equivalent to a Segal category. 

For sketching simplicial categories, we need only the Segal maps $s_k$ for $k\leq 3$. Since such a truncation does not seem to be possible
for Segal categories, we get another example of a normal finite product sketch without a rigidification. 

(2) Similarly, we can treat Tamsamani 2-categories which correspond to Segal categories when we replace $\SSet$ by $\Cat$. In the same way
as in (1), \cite{Bo} implies that any Tamsamani 2-category is equivalent to a 2-category.
}
\end{exam} 

\begin{defi}\label{def5.5}
{\em
Let $\ch=(\ct,L)$ be a weighted homotopy limit sketch. We say that a $\cv$-functor $A:\ct\to\cv$ is an \textit{easy homotopy model} of $\ch$ 
if the morphisms 
$$
t_l^A:AX_l\to\{G,AD_l\}
$$
are trivial fibrations. 
}
\end{defi}

We are following the terminology of \cite{Si}.

\begin{lemma}\label{le5.6}
$S$ preserves easy homotopy models.
\end{lemma}
\begin{proof}
Let $A$ be an easy homotopy model of $\ch$. We have to show that
$$
t^{S(A)}:S(A)\{G,D\}\to\{G,S(A)D\}
$$
is a trivial fibration for each $G:\cd\to\cv$ and $D:\cd\to\ct^\ast$. Following \ref{re4.5}, 
$$
t^{S(A)}:\tilde{A}\overline{J}\{G,D\}\to\{G,\tilde{A}\overline{J}D\}\cong\tilde{A}\{G,\overline{J}D\}.
$$
Since the weighted limit $\{G,D\}$ in the coreflective full subcategory $\ct^\ast$ of $\tilde{\ct}$ is calculated as the coreflection
of the weighted limit $\{G,D\}$ in $\tilde{\ct}$, 
$$
t^{S(A)}=\tilde{A}\varepsilon_{\{G,\overline{J}D\}}
$$
where $\varepsilon:\overline{J}\overline{P}\to\Id$ is the counit. Since $\varepsilon_X$ is fibrantly generated by $\varphi_l$ for $l\in L$
and $\tilde{A}$ preserves all limits, $\tilde{A}_{\{G,D\}}$ is fibrantly generated by $\tilde{A}\varphi_l$, $l\in L$. Since 
$\tilde{A}\varphi_l=t^A_l$, $\tilde{A}_{\{G,D\}}$ is a trivial fibration. Thus $S(A)$ is an easy homotopy model of $\ch^\ast$.
\end{proof}

\begin{coro}\label{cor5.7} 
Let $\cv$ be a finitely combinatorial monoidal model ca\-te\-go\-ry having all objects cofibrant and equipped with a fibrant approximation $\cv$-functor 
$R:\cv\to\cv$ preserving finite weighted limits. Let $\ch=(\ct,L)$ a normal finite weighted homotopy limit sketch. Then each easy homotopy model 
of $\ch$ is weakly equivalent in $[\ct,\cv]$ to a model of $\ch$.
\end{coro}
\begin{proof}
Let $A$ be an easy homotopy model of $\ch$. Following \ref{le5.6}, $S(A)$ is a homotopy model of $\ch^\ast$. Following \ref{th3.4} and \ref{th5.1},
$S(A)$ is weakly equivalent to a model $B$ of $\ch^\ast$. Thus $A$ is weakly equivalent to the model $U_1B$ of $\ch$. 
\end{proof}

\begin{rem}\label{re5.8}
{
\em
A consequence of \ref{cor5.7} is that a normal finite weighted homotopy limit sketch $\ch$ admits a rigidification if and only if it admits
an \textit{easyfication}, i.e., if and only if each homotopy model of $\ch$ is weakly equivalent to an easy homotopy model of $\ch$.
The same is true in a more general context of \ref{th3.3}.
}
\end{rem}

\end{document}